\documentclass[12pt]{amsart}

 \usepackage{mathrsfs}
 \usepackage{amsrefs}
 \usepackage{amsthm}
 \usepackage{enumerate}
 \usepackage{graphicx}
\usepackage{graphicx} 
\usepackage{scrextend}
\usepackage{amsfonts, amsmath, amssymb,amsthm,comment}  
\usepackage{times,enumerate}
\usepackage{mathdots}%
\usepackage{nccmath}
\usepackage{mathtools}
\usepackage{ulem}
\usepackage{enumitem}
\usepackage{multicol}
    {\end{pmatrix}\end{medsize}}%
\usepackage{dsfont,bbm}
\usepackage{rotating}
\usepackage{lscape}
\usepackage{url}
\usepackage{multicol}
\usepackage{thmtools, thm-restate}
\usepackage{blkarray}
\usepackage{multicol}
\usepackage[margin=2.5cm]{geometry}
 \usepackage{hyperref}
\usepackage{cancel}
\usepackage{multirow}
\usepackage[baseline]{euflag}

\usepackage{pifont}
\usepackage{tikz}
\usetikzlibrary{topaths}
\tikzstyle{every picture}+=[remember picture,inner xsep=0,inner ysep=0.25ex]
\usetikzlibrary{calc}
\usepackage{kbordermatrix}
\renewcommand{\kbldelim}{(}
\renewcommand{\kbrdelim}{)}
\renewcommand{\kbrowstyle}{\displaystyle}
\renewcommand{\kbcolstyle}{\displaystyle}
\def\VR{\kern-\arraycolsep\strut\vrule &\kern-\arraycolsep}
\def\vr{\kern-\arraycolsep & \kern-\arraycolsep}
\makeatletter
\newcommand*{\sublabel}[1]{%
    \let\old@currentlabel\@currentlabel%
    \renewcommand{\@currentlabel}{\theenumii}%
    \label{#1}%
    \let\@currentlabel\old@currentlabel%
}
\makeatother
\allowdisplaybreaks

\graphicspath{ {./figs/} }

\hypersetup{
    colorlinks,
    linkcolor={blue},
    citecolor={blue},
    urlcolor={blue}
}

\DeclareMathOperator{\Rank}{rank}

\DeclareMathOperator{\Card}{card}

\DeclareMathOperator{\Pos}{Pos}
\DeclareMathOperator{\fa}{fa}
\DeclareMathOperator{\iso}{iso}
\DeclareMathOperator{\odd}{odd}
\DeclareMathOperator{\even}{even}
\DeclareMathOperator{\Int}{int}
\DeclareMathOperator{\proj}{pr}
\DeclareMathOperator{\ev}{ev}

\makeatletter
\def\widebreve{\mathpalette\wide@breve}
\def\wide@breve#1#2{\sbox\z@{$#1#2$}%
     \mathop{\vbox{\m@th\ialign{##\crcr
\kern0.08em\brevefill#1{0.8\wd\z@}\crcr\noalign{\nointerlineskip}%
                    $\hss#1#2\hss$\crcr}}}\limits}
\def\brevefill#1#2{$\m@th\sbox\tw@{$#1($}%
  \hss\resizebox{#2}{\wd\tw@}{\rotatebox[origin=c]{90}{\upshape(}}\hss$}
\makeatletter

\newcommand{\RR}{\mathbb R}

\newcommand{\dd}{{\rm d}}

\newcommand{\NN}{\mathbb N}
\newcommand{\ZZ}{\mathbb Z}

\newcommand{\bdx}{\mathbf x}

\newcommand{\cS}{\mathcal{S}}

\newcommand{\cI}{\mathcal I}
\newcommand{\cR}{\mathcal R}
\newcommand{\cL}{\mathcal L}

\newcommand{\cM}{\mathcal M}

\newcommand{\cH}{\mathcal H}

\newcommand{\cZ}{\mathcal Z}
\newcommand{\cC}{\mathcal C}

\newcommand{\benu}{\begin{enumerate}}
\newcommand{\eenu}{\end{enumerate}}
\newcommand{\bop}{\begin{opomba}}
\newcommand{\eop}{\end{opomba}}

\newcommand{\supp}{\mathrm{supp}}

\newtheorem{theorem}{Theorem}[section]
\newtheorem{corollary}[theorem]{Corollary}

\newtheorem{proposition}[theorem]{Proposition}

\theoremstyle{definition}

\newtheorem{example}[theorem]{Example}

\newcommand{\mc}{\mathcal}
\newcommand{\mbb}{\mathbb}
\newcommand{\mbf}{\mathbf}

\definecolor{green-new}{rgb}{0.0, 0.5, 0.0}

\definecolor{cyan}{rgb}{0.0, 0.8, 1.0}

\theoremstyle{remark}
\newtheorem{remark}[theorem]{Remark}

\numberwithin{equation}{section}

\begin{document}

\numberwithin{equation}{section}

\title[The truncated univariate rational moment problem]
{The truncated univariate rational moment problem}

\author[R. Nailwal]{Rajkamal Nailwal}
\address{Rajkamal Nailwal, Institute of Mathematics, Physics and Mechanics, Ljubljana, Slovenia.}
\email{rajkamal.nailwal@imfm.si}

\author[A. Zalar]{Alja\v z Zalar}
\address{Alja\v z Zalar, 
Faculty of Computer and Information Science, University of Ljubljana  \& 
Faculty of Mathematics and Physics, University of Ljubljana  \&
Institute of Mathematics, Physics and Mechanics, Ljubljana, Slovenia.}
\email{aljaz.zalar@fri.uni-lj.si}
\thanks{The second-named author was supported by the ARIS (Slovenian Research and Innovation Agency)
research core funding No.\ P1-0228 and grant No.\ J1-50002.}

\begin{abstract}
Given a closed subset $K$ in $\RR$, the rational $K$--truncated moment problem ($K$--RTMP) asks to characterize the existence of a positive Borel measure $\mu$, supported on $K$, such that a linear functional $\cL$, defined on all rational functions of the form $\frac{f}{q}$, where $q$ is a fixed polynomial with all real zeros of even order and $f$ is any real polynomial of degree at most $2k$, is an integration with respect to $\mu$.
The case of a compact set $K$ was solved in \cite{Cha94},
but there is no argument that ensures that $\mu$ vanishes on all real zeros of $q$. An obvious necessary condition for the solvability of the $K$--RTMP is that $\cL$ is nonnegative on every $f$ satisfying $f|_{K}\geq 0$. If $\cL$ is strictly positive on every $0\neq f|_{K}\geq 0$, we add the missing argument from \cite{Cha94} and also bound the number of atoms in a minimal representing measure. We show by an example that nonnegativity of $\cL$ is not sufficient and add the missing conditions to the solution. We also solve the $K$--RTMP for unbounded $K$ and derive the solutions to the strong truncated Hamburger moment problem and the truncated moment problem on the unit circle as special cases.
 \looseness=-1
\end{abstract}

\subjclass[2020]{Primary 47A57, 47A20, 44A60; Secondary 
15A04, 47N40.}

\keywords{Truncated moment problem, rational truncated moment problem, representing measure, moment matrix, localizing moment matrix, preordering.}
\date{\today}
\maketitle

\section{Introduction}
\label{statement-K-RTMP}

Let 
$\RR[x]_{\leq k}:=\{f\in \RR[x]\colon \deg f\leq k\}$ 
stand for the set of real univariate polynomials of degree at most $k$. 
Let $K\subseteq \RR$
be a closed set in $\RR$,
$\lambda_1,\ldots,\lambda_p$
distinct real numbers 
with $p\in \NN\cup\{0\}$,
$\eta_{1},\ldots,\eta_{r}$
distinct positive real numbers 
with
$r\in \NN\cup \{0\}$,
$k_0,\ldots,k_p\in \NN$,
$\ell_1,\ldots,\ell_r\in \NN$,
\begin{equation}
\label{def-of-q}
q(x):=\prod_{j=1}^p(x-\lambda_j)^{2k_j}\cdot 
                \prod_{j=1}^r(x^2+\eta_j)^{\ell_j},
\end{equation}
$2k:=\sum_{j=0}^p 2k_j+\sum_{j=1}^{r}2\ell_j$
and
\begin{equation}
\label{def-of-R}
\cR^{(2k)}
=
\left\{
\frac{f}{q}
                \colon
                f\in \RR[x]_{\leq 2k}
\right\}.
\end{equation}
The \textbf{rational $K$--truncated moment problem ($K$--RTMP)} asks to characterize the existence of a positive Borel measure $\mu$, supported on $K$, such that a linear functional $\cL:\cR^{(2k)}\to \RR$ has an integral representation
\begin{equation}
\label{functional-rational}
\cL(R)=\int_{K}R(x)\dd\mu(x)\quad \forall R\in \cR^{(2k)}.
\end{equation}
If $\mu$ is such a measure, then we call it a \textbf{$K$--representing measure ($K$--rm)} for $\cL$.
The points $\lambda_j$ (resp.\ $\pm i\sqrt{\eta_j}$) are called \textbf{real poles} (resp.\ \textbf{complex poles}). 

\begin{remark}
    An equivalent formulation of the problem \cite{Cha94}, more common in the moment problem literature, is the following. 
    Assume the notation above.
    Given sequences
$\{\gamma_{i}^{(j)}\}_{i=0}^{2k_j}\subset \RR$
where $j=0,\ldots,p$ and $k_j\in \NN\cup \{0\}$,
and sequences
$\{\gamma_{i}^{(p+j,s)}\}_{i=1}^{\ell_j}$
where $j=1,\ldots,r$, $s=0,1$, and $\ell_j\in \NN\cup\{0\},$ characterize the existence of a positive Borel measure $\mu$, supported on $K$, such that 
\begin{align}
\label{rational-MP}
\begin{split}
\gamma_{i}^{(0)}
    &=\int_Kx^i \dd\mu(x),\quad i=0,\ldots,2k_0,\\
\gamma_{i}^{(j)}
    &=\int_K \frac{1}{(x-\lambda_j)^i} \dd\mu(x),
            \quad 
            j=1,\ldots,p, \;
            i=1,\ldots,2k_j,\\
\gamma_{i}^{(p+j,0)}
    &=\int_K \frac{1}{(x^2+\eta_j)^i} \dd\mu(x),
            \quad 
            j=1,\ldots,r,\; 
            i=1,\ldots,\ell_j,\\
\gamma_{i}^{(p+j,1)}
    &=\int_K \frac{x}{(x^2+\eta_j)^i} \dd\mu(x),
            \quad 
            j=1,\ldots,r,\; 
            i=1,\ldots,\ell_j.
\end{split}
\end{align}
Equivalence of the formulations \eqref{functional-rational} and \eqref{rational-MP} is due to partial fractions decomposition, i.e., every $R\in \cR^{(2k)}$ can be expressed as 
$$
\sum_{i=0}^{2k_0}a_ix^{i}+
\sum_{j=1}^{p}\sum_{i=1}^{2k_j}\frac{b_{i,j}}{(x-\lambda_j)^i}+
\sum_{j=1}^{r}\sum_{i=1}^{\ell_j}\frac{c_{i,j}}{(x^2+\eta_j)^i}+
\sum_{j=1}^{r}\sum_{i=1}^{\ell_j}\frac{d_{i,j}x}{(x^2+\eta_j)^i}$$
for  some 
     $a_i, b_{i,j}, c_{i,j}, d_{i,j}\in \RR$.
     \hfill$\vartriangle$
\end{remark}

The univariate rational moment problems for the interval (bounded or unbounded),
especially the full version without bounds on the degrees of moments (i.e., $k_j$ can be $\infty$
in \eqref{rational-MP} above) were studied extensively by Jones, Nj\r{a}stad and Thron \cite{JTW80,JT81,JNT84,Nja85,NT86,Nja87,Nja88}. 
The main approach for their results was a theory of orthogonal and quasi-orthogonal rational functions.
For any compact set $K\subset \RR$
the $K$--RTMP was studied by Chandler \cite{Cha94} using duality with positive polynomials.
The motivation for this paper was to extend Chandler's solution from a compact set $K$, in which only real poles are allowed, to an arbitrary closed set $K\subseteq \RR$, in which complex poles are also present. Apart from the fact that the $K$--RTMP is interesting for its own sake, the application of the univariate reduction technique 
can also provide solutions to the bivariate truncated moment problem (TMP) on some algebraic curves, where all irreducible components are rational
(see \cite[Section 6]{BZ21} for $xy=0$; \cite{Zal22a} for $y^2=y$ and $y(y-a)(y-b)=0$, $a,b\in \RR\setminus\{0\}$, $a\neq b$;
\cite{Zal21} for $y=x^3$ and $y^2=x^3$;
 \cite{Zal22b} for $xy=1$ and $xy^2=1$; \cite{Zal23} for $y=q(x)$ and $yq(x)=1$, $q\in \RR[x]$; \cite{YZ24} for $y(ay+x^2+y^2)=0$, $a\in \RR\setminus \{0\}$, and $y(x-y^2)=0$).

We also mention that versions of the multidimensional rational moment problem 
for linear functionals on localizations of the polynomial algebra have been 
investigated in \cite{PV99,CMN11,Sch16}.

 A technique used in \cite{Cha94} to solve the $K$--RTMP is to convert the problem into the usual TMP, which concerns the integral representability of linear functionals on univariate polynomials of bounded degree with respect to the measure supported on $K$. This simplifies the problem since univariate TMP\textit{s} have been widely studied in the literature \cite{Akh65,KN77,Ioh82,CF91}, but one must additionally characterize when the measure vanishes on all real poles of the $K$--RTMP. This detail is not taken care of in \cite{Cha94}. Our first main result closes this gap for strictly positive functionals on $K$, i.e., functionals that are positive on every nonzero polynomial nonnegative on $K$. We extend the result to arbitrary closed sets $K$, where complex poles are allowed. We formulate the result in terms of the corresponding functional on polynomials, where at most countable closed set   in $K$ is to be avoided by the measure.
The proof is done by applying a more general result of di Dio and Schm\"udgen (see 
\cite[Proposition 2]{DS18} or \cite[Theorem 1.30]{Sch17}), which holds for strictly positive linear functionals on arbitrary finite-dimensional subspaces of the real vector space of continuous functions on a locally compact Hausdorff space.
\cite[Corollary 6]{DS18} also provides an upper bound for the Caratheodory number, i.e., the number of atoms needed in a minimal representing measure. Moreover, we provide a constructive proof in the case $K$ is a closed semialgebraic set, which also improves the upper bound on the Caratheodory number obtained by applying \cite[Corollary 6]{DS18}. In the proof, we essentially use a result of Blekherman et al \cite{BKRSV20}, which characterizes minimal quadrature rules for linear functionals on univariate polynomials of bounded degree.
Our second main result solves the $K$--RTMP for arbitrary closed $K$ for positive functionals that are not strictly positive, i.e., the functional can vanish on some nonzero polynomial that is nonnegative on $K$.
We construct a counterexample to the solution \cite[Proposition 2]{Cha94} for compact $K$, which misses additional conditions except $K$--positivity of the functional.
Finally, we apply our main results to obtain a solution to the strong Hamburger TMP \cite{Zal22b} and the TMP on the unit circle \cite{CF02}.

\subsection{Reader's guide} 

The paper is structured as follows.
In Section \ref{sec:prel} we introduce some further notation, show the correspondence between the $K$--RTMP and the corresponding univariate $K$--TMP, recall the result characterizing positive polynomials on $K$, the notions of localizing Hankel matrices, the solution to the $\RR$--TMP by Curto and Fialkow, and the characterization of minimal quadrature formulas in the nonsingular case by Blekherman et al.
In Section \ref{sec:main-results} we solve the $K$--TMP coming from the $K$--RTMP both in the nonsingular case (see Theorem \ref{thm:nonsingular}) and in the singular case (see Theorem \ref{compact-TRMP-extended}). Example \ref{counterexample} shows that $K$--positivity of the functional is not sufficient for the existence of a $K$--rm that avoids real poles.
Finally, in Section \ref{sec:examples} we derive the solutions to the strong Hamburger TMP (see Corollary \ref{cor:STHMP}) and the TMP on the unit circle (see Theorem \ref{circle-TMP}), and give an example (see Example \ref{counterexample-2}), which demonstrates the construction of the measure as in the proof of Theorem \ref{thm:nonsingular}.

\section{Preliminaries}
\label{sec:prel}

We write $\RR^{n\times m}$ for the set of $n\times m$ real matrices. For a matrix $M$ 
we call the linear span of its columns a \textbf{column space} and denote it by $\cC(M)$.
The set of real symmetric matrices of size $n$ will be denoted by $S_n$. 
For a matrix $A\in S_n$ the notation $A\succ 0$ (resp.\ $A\succeq 0$) means $A$ is \textbf{positive definite (pd)} (resp.\ \textbf{positive semidefinite (psd)}).

For a polynomial $f\in \RR[x]$
we denote by $\cZ(f):=\{x\in \RR\colon f(x)=0\}$ its set of zeros.


\subsection{Representing measures}
Assume the notation from \S\ref{statement-K-RTMP}. 
For $\cL:\cR^{(2k)}\to \RR$ we define a corresponding 
linear functional $L$ on $\RR[x]_{\leq 2k}$ by
\begin{equation}
    \label{Riesz-functional}  
    L:\RR[x]_{\leq 2k}\to \RR,\quad L(f):=\cL(fq^{-1}).
\end{equation}
We call a positive Borel measure  $\mu$, supported on a closed set $K\subseteq \RR$, a \textbf{$K$--representing measure ($K$--rm)} for $L$
if and only if 
\begin{equation*}
    L(f)=
    \int_K f\dd\mu\quad \forall f\in \RR[x]_{\leq 2k}.
\end{equation*}

For $x\in \RR$, $\delta_x$ 
stands for the Dirac measure supported on $x$.
By a \textbf{finitely atomic positive measure} on $\RR$ we mean a measure of the form $\mu=\sum_{j=0}^\ell \rho_j \delta_{x_j}$, 
where $\ell\in \NN$, each $\rho_j>0$ and each $x_j\in \RR$. The points $x_j$ are called 
\textbf{atoms} of the measure $\mu$ and the constants $\rho_j$ the corresponding \textbf{densities}.

Let $\Lambda\subseteq K$ be a set. We write
\begin{align*}
    \cM_{\cL,K}
    &:=
\{\mu\colon
\mu \text{ is a $K$--representing measure for $\cL$}\},\\
    \cM_{L,K}
    &:=
\{\mu\colon
\mu \text{ is a $K$--representing measure for $L$}\},\\
    \cM_{L,K, \Lambda}
    &:=
\{\mu\colon
\mu \text{ is a $K$--representing measure for $L$ with }\mu(\Lambda)=0\}.
\end{align*}
We denote by 
$\cM^{(\fa)}_{\cL,K}$, $\cM^{(\fa)}_{L,K}$ and
$\cM^{(\fa)}_{L,K,\Lambda}$
the subsets of $\cM_{\cL,K}$, $\cM_{L,K}$ and
$\cM_{L,K,\Lambda}$, respectively,
containing all finitely atomic measures.

Let $\mu$ be a Borel measure supported on $K$.
Let $f$ be a $\mu$-integrable functions. We denote by $f\cdot \mu$
a Borel measure on $K$, defined by 
\begin{equation} 
    \label{def:new-measure}
        (f\cdot \mu)(E):=\int_E f \dd\mu
\end{equation}
for every Borel set $E\subseteq K$.
\begin{proposition}
\label{cL-and-L-measures}
Let $q$ be as in \eqref{def-of-q}.
The following statements hold:
\begin{enumerate}
\item 
\label{equivalence-between-the-problems}
$\cM_{\cL,K}\neq \emptyset$
if and only if
$\cM_{L,K, \cup_{j=1}^p \{\lambda_j\}}
\neq \emptyset$.
\item 
\label{equivalence-between-the-problems-2}
$\cM_{\cL,K}^{(\fa)}\neq \emptyset$
if and only if
$\cM^{(\fa)}_{L,K, \cup_{j=1}^p \{\lambda_j\}}
\neq \emptyset$.
\item 
\label{phi-bijection}
A map
$$\Phi:\cM_{L,K, \cup_{j=1}^p \{\lambda_j\}}\to \cM_{\cL,K},
\quad \mu\mapsto q\cdot \mu,$$ 
is a bijection.
The inverse of $\Phi$ is 
$\Phi^{-1}(\mu)=\frac{1}{q}\cdot \mu$.
\end{enumerate}
\end{proposition}

\begin{proof}
Note that \eqref{equivalence-between-the-problems}
and \eqref{equivalence-between-the-problems-2}
follow from \eqref{phi-bijection}.
So it suffices to prove \eqref{phi-bijection}.

Let $\mu\in \cM_{L,K, \cup_{j=1}^p \{\lambda_j\}}$
and $\frac{f}{q}\in \cR^{(2k)}$. We have
$$
\cL\Big(\frac{f}{q}\Big)
\stackrel{\eqref{Riesz-functional}}{=}
L(f)
=
\int_{K}f\dd\mu
=\int_K \frac{f}{q}q \dd\mu
=\int_K \frac{f}{q} \dd(q\cdot \mu)
,
$$
where the third equality is well--defined since $\mu(\cZ(q))=0$.
Hence, $q\cdot\mu\in \cM_{\cL,K}$. 

Conversely, let $\mu\in \cM_{\cL,K}$
and $f\in \RR[x]_{\leq 2k}$. We have
$$
L(f)
\stackrel{\eqref{Riesz-functional}}{=}
\cL\Big(\frac{f}{q}\Big)
=
\int_{K}\frac{f}{q} \dd\mu
=\int_K f\frac{1}{q} \dd\mu
=\int_K f \dd\big(\frac{1}{q}\cdot \mu\big).
$$
Hence, $\frac{1}{q}\cdot\mu\in \cM_{L,K,\cup_{j=1}^p \{\lambda_j\}}$. Clearly, $\mu(\{\lambda_j\})=0$ for each $j$ otherwise $\int_K \frac{1}{q}\dd\mu$ was not well--defined. 
\end{proof}

\begin{remark}
In \cite[p.\ 75]{Cha94} it is claimed that each $\mu\in \cM_{L,K}$ yields 
$q\cdot \mu\in \cM_{\cL,K}$. This is wrong since $\mu(\{\lambda_j\})$ could be nonzero in some $\lambda_j$. The $K$--RTMP in \cite{Cha94} for compact $K$ is solved under this claim. In the next section we solve the $K$--RTMP for any closed set $K$ using correct correspondence between measures for $L$ and $\cL$.
\hfill$\vartriangle$
\end{remark}


\subsection{Positive polynomials}

We denote by 
$$\Pos(K):=\{f\in \RR[x]\colon f(x)\geq 0\text{ for all }x\in K\}$$ 
the set 
of all polynomials, nonnegative on $K$.
Let 
$$\Pos_{\leq 2k}(K):=\Pos(K)\cap\RR[x]_{\leq 2k}.$$
We denote by $\sum \RR[x]^2$ (resp.\ $\sum \RR[x]^2_{\leq k}$) the set of all finite sums
of squares $p^2$ of polynomials, where $p\in \RR[x]$ 
(resp.\ $p\in\RR[x]_{\leq k}$).\\

We call a linear functional $L:\RR[x]_{\leq 2k}\to \RR$:
\begin{enumerate}
    \item 
    \textbf{$K$--positive}, if $L(f)\geq 0$ for every
        $f\in \Pos_{\leq 2k}(K)$.
    \item 
    \textbf{strictly $K$--positive}, if it is positive and $L(f)>0$ for every $0\neq f\in \Pos_{\leq 2k}(K)$.
    \item 
    \textbf{square--positive}, if $L(g)\geq 0$ for every $g\in\sum \RR[x]^2_{\leq k}$.
    \item \textbf{singular},
    if $L(g^2)=0$ for some $0\neq g$ such that $g^2\in \RR[x]_{\leq 2k}$.
\end{enumerate}

\subsection{Preordering and the natural description}
Given a finite set $S:=\{g_1,g_2,\ldots,g_n\}$ or a countable set $S:=\{g_i\}_{i=1}^{\infty}$ in $\RR[x]$ and $e:=(e_1,\ldots,e_m)\in \{0,1\}^m$,
let  $\underline{g}^e$ stand for $g_1^{e_1}g_2^{e_2}\cdots g_{m}^{e_m}.$
Let 
$$E:=\left\{
    \begin{array}{rl}
        \{0,1\}^n,& \text{if }S\text{ has }n\text{ elements},\\[0.2em]
        \cup_{j=1}^\infty \{0,1\}^j,&   \text{if }S\text{ is infinite},
    \end{array}
\right.$$
and 
    $$S^{\pi}:=\left\{\underline{g}^e:e \in E\right\}.$$
The \textbf{preordering generated by $S$} in $\RR[x]$ is defined by
	\begin{eqnarray*}
		T_{S} 
			&:=& 
				\Big\{\sum_{
                s\in S^\pi} \sigma_s s\colon \sigma_s\in 
				\sum \RR[x]^2\;\text{for each}\;s\text{ and }\sigma_s\neq 0\text{ for finitely many }s\Big\}.
	\end{eqnarray*}
For $d\in \NN\cup\{0\}$ we define the set 
	\begin{align*} 
		T^{(d)}_{S}
            &:=
            \Big\{\sum_{s\in S^\pi} 
            \sigma_s s\colon 
			\sigma_s \in \sum\RR[x]^2
            \text{ and }\deg(\sigma_s s)\leq d\; \text{for each}\;s, \;
            \sigma_s\neq 0\text{ for finitely many }s\Big\}.
	\end{align*}
	We call $T^{(d)}_{S}$ the \textbf{degree $d$ truncation} of the preordering $T_{S}$.\\

A set $S_K\subset \RR[x]$ is the
\textbf{natural description} of the closed set $K$, if it satisfies the following conditions:
	\begin{enumerate}
		\item[(a)] If $K$ has the least element $a\in\RR$, then $x-a\in S_K$.
		\item[(b)] If $K$ has the greatest element $b\in \RR$, then $b-x\in S_K$.
		\item[(c)] 
  For every $a,b\in K$, $a\neq b$, if $(a,b)\cap K=\emptyset$, then $(x-a)(x-b)\in S_K$.
		\item[(d)] These are the only elements of $S_K$.
	\end{enumerate}

\begin{remark}
The definition of the natural description coincides with the one given in \cite[2.3 \textit{Notes}.(2)]{KM02} only that we allow any closed set $K$, not necessarily a semialgebraic one, i.e., a union of finitely many closed intervals.\hfill$\vartriangle$
\end{remark}

For a closed set $K\subseteq \RR$ we write $\cI(K)$ to denote
the smallest closed interval containing $K$.
Note that $\cI(K)\setminus K$ is of the form 
     $$\cI(K)\setminus K=\cup_{i\in \Gamma}(a_i,b_i),$$
   where $\{(a_i,b_i)\colon i\in \Gamma\}$ is a family of
   pairwise disjoint bounded intervals and the index set $\Gamma$
   is at most countable.
   For $J\subseteq \Gamma$ we define the set 
   $$K_J=\cI(K)\setminus \cup_{j\in J}(a_j,b_j)$$

\begin{proposition}
\label{prop:natural-desc}
     Let $K\subseteq \RR$ be a closed set
     and
     $\cI(K), a_i,b_i,\Gamma, K_J$ defined as above.
     Let $\Omega$ be a set of all finite subsets of $\Gamma$.
     The following statements hold:
     \begin{enumerate}
        \item  
        \label{prop:natural-desc-pt1}
        $\displaystyle\Pos_{\leq d} ({K})=
        \underset{J \in\Omega}\cup
        \Pos_{\leq d} (K_J).$
        \item
        \label{prop:natural-desc-pt2}
            $\displaystyle T_{S_K}^{(d)}=\underset{J \in \Omega}{\cup}T^{(d)}_{S_{K_J}}.$
        \item 
        \label{prop:natural-desc-pt3}
        $\Pos_{\leq d}(K)=T^{(d)}_{S_K}$.
     \end{enumerate}
 \end{proposition}
 
 \begin{proof}
First we prove \eqref{prop:natural-desc-pt1}.
Since $K\subseteq K_J$ for every $J\in \Omega$, the inclusion $(\supseteq)$ is trivial. To prove the  inclusion $(\subseteq)$
take $p\in \Pos_{\leq d}(K)$.
Note that $p$ is nonnegative on all but at most finitely many intervals 
   $(a_{i_1},b_{i_1}),\ldots, (a_{i_j},b_{i_j})$,
   where $i_\ell\in \Gamma$ for each $i_\ell$. This follows from the observation that if $p$ is negative in a point from $(a_i,b_i)$,
   then it should have a zero on $(a_i,b_i)$ in order to be nonnegative on $K$. Since the degree of $p$ is at most $d$,
   there are at most $d$ disjoint intervals $(a_i,b_i)$, where
   $p$ could have a zero. But then $p\in  \Pos_{\leq d} (K_{J})$
   for $J=\{i_1,\ldots,i_j\}$.
   
\eqref{prop:natural-desc-pt2} follows by noticing that 
$S_K^{\pi}=\underset{J \in \Omega} \cup S_{K_J}^{\pi}$. 

It remains to prove \eqref{prop:natural-desc-pt3}.
By \cite[Theorem 4.1]{KMS05}, we have that
$T^{(d)}_{S_{K_J}}=\Pos_{\leq d}(K_J)$ for every $J\in \Omega$.
This fact, together with \eqref{prop:natural-desc-pt1} and \eqref{prop:natural-desc-pt2}, implies \eqref{prop:natural-desc-pt3}.
 \end{proof}
 

\subsection{Localizing Hankel matrices}
\label{Hankel-mat-ext}

Let $\gamma\equiv\gamma^{(2k)}=(\gamma_{0},\gamma_{1},\ldots,\gamma_{2k})\in \RR^{2k+1}$
be a sequence.
For $\ell\in \NN$, $\ell\leq k$,
the Hankel matrix $\cH_{1,\gamma^{(2\ell)}}$ of size $(\ell+1)\times (\ell+1)$,
with columns and rows indexed by the monomials $\textit{1},X,\ldots,X^\ell$,
is equal to	
\begin{equation}\label{vector-v}
		\cH_{1,\gamma^{(2\ell)}}:=\left(\gamma_{i+j} \right)_{i,j=0}^\ell
					=	\kbordermatrix{
							& \textit{1} & X & X^2 & \cdots  & X^{\ell} \\
							\textit{1} & \gamma_0 & \gamma_1 & \gamma_2 & \cdots &\gamma_\ell\\
							X & \gamma_1 & \gamma_2 & \iddots & \iddots & \gamma_{\ell+1}\\
							X^2 & \gamma_2 & \iddots & \iddots & \iddots & \vdots\\
							\vdots &\vdots 	& \iddots & \iddots & \iddots & \gamma_{2\ell-1}\\
							X^\ell & \gamma_\ell & \gamma_{\ell+1} & \cdots & \gamma_{2\ell-1} & \gamma_{2\ell}
						}.
	\end{equation}
\textbf{Convention:} If $\gamma\equiv\gamma^{(2\ell+1)}=(\gamma_{0},\gamma_{1},\ldots,\gamma_{2\ell+1})\in \RR^{2\ell+2}$
is of even length, then we define 
$\cH_{1,\gamma}:=\cH_{1,\gamma^{(2\ell)}}$, i.e., we omit the last coordinate.\\

For $p(x)=\sum_{i=0}^{\ell'} a_{i}x^i\in \RR[x]$, $\ell'\leq \ell$, 
we define the \textbf{evaluation} $p(X)$ on the columns of the matrix $\cH_{1,\gamma^{(2\ell)}}$ by replacing each capitalized monomial $X^i$ by the column of $\cH_{1,\gamma^{(2\ell)}}$, indexed by this monomial.
Then $p(X)$ is a vector from the linear span of the columns of $\cH_{1,\gamma^{(2\ell)}}$. 
If this vector is the zero one,
then we say $p$ is a \textbf{column relation} of $\cH_{1,\gamma^{(2\ell)}}$.

Let $\cH_{1,\gamma^{(2k)}}$ be psd and singular. Let $\ell\in\NN$ be the smallest number such
that $\cH_{1,\gamma^{(2\ell)}}$ is singular. Then the only column relation of $\cH_{1,\gamma^{(2\ell)}}$ 
is of the form
$X^\ell=a_0\textit{1}+a_1X+\ldots+a_{\ell-1}X^{\ell-1}$ for some $a_i\in \RR$.
The polynomial
$$p_\gamma(x)=x^\ell-\sum_{i=0}^{\ell-1}a_ix^i\in\RR[x]_{\leq \ell}$$
is called
the \textbf{generating polynomial} of $\gamma$. 
Let 
$$
q_{i,\gamma}(x):=x^{i}\cdot p_\gamma(x), \;i\in \NN.$$
By \cite[Theorem 2.4]{CF91}, all polynomials 
$q_{1,\gamma},q_{2,\gamma},\ldots, q_{k-\deg p_\gamma-1,\gamma}$ are column relations of 
$\cH_{1,\gamma^{(2k)}}$,
while $q_{k-\deg p_\gamma,\gamma}$ being a column relation or not determines the existence of a $\RR$--rm for $\gamma$ (see Theorem \ref{Hamburger} below).

The \textbf{Riesz functional} $L_\gamma:\RR[x]_{\leq 2k}\to \RR$ of $\gamma$ is defined by $L_\gamma(x^i):=\gamma_i$ for each $i$.
Let $f\in \RR[x]_{\leq 2k}$. An \textbf{$f$--localizing Hankel matrix $\cH_{f,\gamma}$ of $\gamma$} is a real square matrix of size 
$s(k,f)\times s(k,f)$, where 
    $s(k,f)=k+1-\lceil\frac{\deg f}{2}\rceil$,
with the $(i,j)$--th entry equal to $L_\gamma(fx^{i+j-2})$.
We write
$$f\cdot \gamma:=(\gamma^{(f)}_0,\gamma^{(f)}_1,\ldots,
\gamma^{(f)}_{2k-\deg f}),
\quad
\gamma^{(f)}_i:=L_\gamma(fx^i).$$
Note that $\cH_{f,\gamma}=\cH_{1,f\cdot\gamma}$. We denote the Riesz functional of $f\cdot \gamma$ by $L_{f,\gamma}$ and call it an \textbf{$f$--localizing Riesz functional of $\gamma$}.

For a functional $L:\RR[x]_{\leq 2k}\to \RR$ the notation $L_f$ stands for $L_{f,\gamma}$, where $\gamma$ is a sequence belonging to $L$, i.e., $\gamma_i:=L(x^i)$ for each $i$.

\begin{proposition}
The following statements are equivalent:
\begin{enumerate}
\item 
\label{part-1-prop3}
$L_{f,\gamma}$ is square--positive.
\item 
\label{part-2-prop3}
$L_{f,\gamma}(g)\geq 0$ for every $g\in \sum\RR[x]^2$ such that $\deg(g)\leq 2k-\deg{f}$.
\item 
\label{part-3-prop3}
$H_{f,\gamma}$ is positive semidefinite.
\end{enumerate}
\end{proposition}

\begin{proof}
The equivalence $\eqref{part-1-prop3}\Leftrightarrow\eqref{part-2-prop3}$
is clear.
The equivalence $\eqref{part-1-prop3}\Leftrightarrow\eqref{part-3-prop3}$
follows by the equality
$L_{f,\gamma}(g^2)=(\hat g)^T\cH_{f,\gamma} \hat g$,
where $\hat g$ is a vector of coefficients of $g$ in the order $1,x,\ldots,x^{k-\lceil\frac{\deg f}{2}\rceil}$.
\end{proof}

\subsection{Solution to the $\RR$--TMP}
Let $x_1,\ldots, x_{r}\in \RR$. We denote by $V_{(x_1,\ldots,x_r)}:=(x_j^{i-1})_{i,j}\in \RR^{r\times r}$ the Vandermondo matrix.
The following is a solution to the $\RR$--TMP.

\begin{theorem}[{\cite[Theorems 3.9 and 3.10]{CF91}}]\label{Hamburger}
	Let $k\in \NN$, $\gamma=(\gamma_0,\ldots,\gamma_{2k})\in \RR^{2k+1}$ with $\gamma_0>0$ and $L_\gamma:\RR[x]_{\leq 2k}\to \RR$ the Riesz functional of $\gamma$.
	The following statements are equivalent:
\begin{enumerate}
	\item
		\label{pt1-130222-1851} 
		$\cM_{L_\gamma,\RR}\neq \emptyset$.
        \item 
            \label{extension}
                There exist  $\gamma_{2k+1}, \gamma_{2k+2}\in \RR$
                such that $\cH_{1,\gamma^{(2k+2)}}$
                is positive semidefinite.
	\item\label{pt4-v2206} 
	One of the following statements holds:
	\begin{enumerate}
		\item $\cH_{1,\gamma^{(2k)}}$ is positive definite.
		\item $\cH_{1,\gamma^{(2k)}}$ is positive semidefinite 
        and if $p_\gamma$ is the generating polynomial of $\gamma$,
        then the polynomial
        $x^{k-\deg p_\gamma}\cdot p_\gamma(x)$ is a column relation of $\cH_{1,\gamma^{(2k)}}$.
            \item 
        $L_\gamma$ is square--positive and if $0\neq p\in \RR[x]$ is a polynomial of the lowest degree such that
        $p^2\in \ker L_\gamma$, then 
        $x^{2k-2\deg p}\cdot p^2\in \ker L_\gamma$.
	 \end{enumerate}
\end{enumerate}
Moreover, if $\cM_{L_\gamma,\RR}\neq \emptyset$, then:
\begin{enumerate}[label=(\roman*)]
	\item\label{140222-1158}
        If $\cH_{1,\gamma^{(2k)}}$
        is singular, 
        then 
        $\cM_{L_\gamma,\RR}=\big\{\sum_{i=1}^{r}\rho_i\delta_{x_i}\big\}$, where $x_1,\ldots,x_r$ are the roots of $p_\gamma$ 
	and
	$
	(\rho_i)_{i=0}^r=	
	V_{(x_1,\ldots,x_r)}^{-1}
        (\gamma_i)_{i=0}^{r-1}.
	$
        In this case there are unique $\gamma_{2k+1}, \gamma_{2k+2}\in \RR$
                such that $\cH_{1,\gamma^{(2k+2)}}$
                is positive semidefinite.
	\item\label{140222-1202} 
        If $\cH_{1,\gamma^{(2k)}}$ is invertible, then there are infinitely many $(k+1)$--atomic measures in 
        $\cM_{L_\gamma,\RR}$.
        They are obtained by choosing $\gamma_{2k+1}\in \RR$ arbitrarily, defining 
	$$\gamma_{2k+2}:=
        \big((\gamma_i)_{i=k+1}^{2k+1}\big)^T
        (\cH_{1,\gamma^{(2k)}})^{-1}
        (\gamma_i)_{i=k+1}^{2k+1}$$ 
		and using
		\ref{140222-1158} for 
		$\widetilde \gamma:=(\gamma_0,\ldots,\gamma_{2k+1},\gamma_{2k+2})\in \RR^{2k+3}.$ 
\end{enumerate}
\end{theorem}

The following result characterizes the existence of a $(k+1)$--atomic measure for a sequence $\gamma\in \RR^{2k+1}$ with $\cH_{1,\gamma}\succ 0$, having one prescribed atom in the support. 

\begin{theorem}[{\cite[Theorem 4]{BKRSV20}}]
	\label{230422-1853}
	Let $k\in \NN$ and $\gamma=(\gamma_0,\ldots,\gamma_{2k})\in\RR^{2k+1}$ be a sequence such that 
	$\cH_{1,\gamma}$ is positive definite.
	The following statements are equivalent:
	\begin{enumerate}
	\item 
	\label{230422-1853-pt1}
		There exists a $(k+1)$--atomic $\RR$--representing measure for $\gamma$ with one of the
		atoms equal to $x_1$.
	\item 
	\label{230422-1853-pt2}
		$x_1\cH_{1,\gamma^{(2k-2)}}-\cH_{x,\gamma}$ is invertible.
	\end{enumerate}

	Moreover, if the equivalent statements 
		\eqref{230422-1853-pt1} and \eqref{230422-1853-pt2} 
	hold, then the other $k$ atoms except $x_1$ are precisely the solutions to
        $g(\mbf{x})=0$,
	where
		$$
		g(\mbf{x}):=\det G(x_1,\mbf{x}),
		\quad
		G(\mbf{x}_1,\mbf{x}_2)=
        \mbf{x}_1\mbf{x}_2\cH_{1,\gamma^{(2k-2)}} 
        -(\mbf{x}_1+\mbf{x}_2) \cH_{x,\gamma}
        + \cH_{x^2,\gamma}.
		$$
\end{theorem}

\section{Solution to the $K$--RTMP}  
\label{sec:main-results}

Let $K\subseteq \RR$ be a closed set, $\Gamma\subseteq \RR$ at most countable closed set
and $L:\RR[x]_{\leq 2k}\to \RR$ a linear functional. In this section we characterize when $L$ has a $K$--rm vanishing on $\Lambda$, i.e., $\cM_{L,K,\Lambda}\neq \emptyset$. By Proposition \ref{cL-and-L-measures}, this in particular solves the $K$--RTMP for $\cL$ on $\cR^{(2k)}$, defined by \eqref{def-of-R}.
The case of nonsingular $L$ is covered by Theorem \ref{thm:nonsingular}, while the case of singular $L$ by Theorem \ref{compact-TRMP-extended}. Example \ref{counterexample} shows that the solution \cite[Proposition 2]{Cha94} for a compact set $K$ is missing additional conditions from Theorem \ref{compact-TRMP-extended} other than $K$--positivity of $L$.\\

We use $\partial K$ and $\Int(K)$ 
to denote the topological boundary and the interior of the set $K$, respectively.
Let $\iso(K)$ be the set of isolated points of $K$.
We use $\Card(V)$ to denote the cardinality of the set $V$. A closed set $K\subseteq \RR$ is \textbf{semialgebraic}, if it is of the form
$$
K:=\left\{x\in \RR\colon p_1(x)\geq 0,\ldots,p_m(x)\geq 0\right\}
$$
for some $p_1,\ldots,p_m\in \RR[x]$.\\

By Proposition \ref{cL-and-L-measures}, solving $K$--RTMP, defined in \S\ref{statement-K-RTMP}, is equivalent to solving the $K$--TMP for $L$, defined by 
\eqref{Riesz-functional}, where the measure must vanish on all real poles. In this section we solve the $K$--TMP for $L$, where the measure has to vanish on a given closed set, which is at most countable. \\

The following is the solution to the nonsingular case of the $K$--TMP for $L$.

\begin{theorem}[Nonsingular case]
\label{thm:nonsingular}
Let $K\subseteq \RR$ be a closed set and $\Lambda\subset \RR$ be a finite or a countable closed set
such that $\Lambda\cap \iso(K)=\emptyset$.
Let $L:\RR[x]_{\leq 2k}\to \RR$ be a linear functional
and 
$L|_{T_{S_K}^{(2k)}\setminus\{0\}}>0$, where $S_K$ is the natural description of $K$.
Then there exist a $r$--atomic measure $\mu\in \cM^{(\fa)}_{L,K,\Lambda}$  
with $k+1\leq r\leq 2k+1$.

Moreover, assume that $K$ is closed and semialgebraic and write 
$\Card(\partial K)=2\ell_1+\ell_2$, $\ell_1\in \NN\cup \{0\}$, $\ell_2\in \{0,1\}$.
Then $r$ is at most:
\begin{enumerate}[label=\roman*)]
\item
\label{thm:nonsingular-part1}
    $k+1$, if $K\in \left\{\RR,[a,\infty),(-\infty,a]\right\}$ for some 
    $a\in \RR$.
    \smallskip
\item
\label{thm:nonsingular-part2} 
    $k+\ell_1+1$, if
    $K$ is bounded and has a non-empty interior.
    \smallskip
\item
\label{thm:nonsingular-part3} 
    $k+\ell_1+\ell_2+1$, if
    $K$ is bounded only from one side. 
    \smallskip
\item
\label{thm:nonsingular-part4}   
    $k+\ell_1+2$, if
    $K$ is unbounded from both sides.
    \smallskip
\end{enumerate}
\end{theorem}

\begin{proof}
Since
$\Lambda\cap \iso(K)=\emptyset$, it follows that 
$\Pos_{\leq 2k}(K)=\Pos_{\leq 2k}(K\setminus \Lambda)$. 
By Proposition \ref{prop:natural-desc},
the assumption $L|_{T_{S_K}^{(2k)}\setminus\{0\}}>0$ implies that
$L|_{\Pos_{\leq 2k}(K)\setminus\{0\}}>0$.
Since $\Lambda$ is closed, $K\setminus \Lambda$ is locally compact Hausdorff space. 
By \cite[Proposition 2 and Corollary 6]{DS18} (or \cite[Theorem 1.30]{Sch17}), $L$ has  a $r$--atomic $K$--rm $\mu$ such that $\mu(\Lambda)=0$ 
and
$r\in \{k+1,\ldots,2k+1\}$. This proves the first part of the theorem.\\

Let us prove the moreover part. 
Assume first that the assumptions of \ref{thm:nonsingular-part2}
hold, i.e., 
$K$ is bounded, closed, semialgebraic set with a non-empty interior. \\

\noindent \textbf{Claim 1:} It suffices to prove  \ref{thm:nonsingular-part2} under the assumption  $\Lambda\cap\partial K=\emptyset$.\\

\noindent \textit{Proof of Claim 1.}
Let $K_1:=\partial K\setminus \iso(K)$.
Since $\Pos_{\leq 2k}(K)=\Pos_{\leq 2k}(K\setminus K_1)$, 
there is $\mu\in \cM_{L,K,\Lambda\cup K_1}^{(\fa)}$ by the first part of the theorem,
in particular $\supp(\mu)\subseteq K\setminus (\Lambda \cup K_1)$.
Let $a:=\min(K)$, $b:=\max(K)$.
Note that 
$$K=[a_0,a_1]\cup [a_2,a_3]\cup \cdots \cup
[a_{2m-2},a_{2m-1}]\cup [a_{2m},a_{2m+1}],$$
where $m\in \NN\cup \{0\}$, $a_0:=a,$ $a_{2m+1}:=b$
and $a_{2i}\leq a_{2i+1}<a_{2i+2}$ for $i=0,\ldots,m-1$ and $a_{2m}\leq a_{2m+1}$.
We possibly shorten each interval $[a_{2i},a_{2i+1}]$ to $[\widetilde a_{2i},\widetilde a_{2i+1}]\subseteq [a_{2i},a_{2i+1}]$
such that
$$\widetilde K:=[\widetilde a_0,\widetilde a_1]\cup [\widetilde a_2,\widetilde a_3]\cup \cdots \cup
[\widetilde a_{2m-2},\widetilde a_{2m-1}]\cup [\widetilde a_{2m},\widetilde a_{2m+1}]$$
has the following properties:
\begin{enumerate}
    \item 
        \label{property-1}
        $\widetilde K$
        contains all atoms
        of $\mu$, 
    \item 
        \label{property-2}
        $\supp(\mu)\cap \Int(K)=\supp(\mu)\cap \Int(\widetilde K)$
    \item 
        $\supp(\mu)\cap \iso(K)=
        \supp(\mu)\cap \iso(\widetilde K)$,
    \item 
        \label{property-3}
        $\Lambda \cap \partial \widetilde K=\emptyset$, 
    \item
        \label{property-4}
            $\Card(\partial \widetilde K)=\Card(\partial K)$. 
\end{enumerate}
To prove Claim 1 it only remains to prove that 
$L|_{\Pos_{\leq 2k}(\widetilde K)\setminus \{0\}}>0$.
Since $\widetilde K$ contains all atoms of $\mu$, it follows that
$L|_{\Pos_{\leq 2k}(\widetilde K)\setminus \{0\}}\geq 0$. 
Assume that there exists $p\in \Pos_{\leq 2k}(\widetilde K)\setminus \{0\}$ such that $L(p)=0$. 
Let $p_0$ be obtained from $p$ by moving each zero of $p$ on $[a_i,\widetilde a_i]$ or $[\widetilde a_i,a_i]$  to $a_i$ together with multiplicity, i.e., every factor $(x-\alpha)^{t}$ of $p$ where $\alpha\in (a_i,\widetilde a_i]$ or 
$\alpha\in [\widetilde a_i,a_i)$ and $t$ is largest possible, is replaced by $(x-a_i)^{t}$. By construction, $0\neq p_0\in \Pos_{\leq 2k}(K)$ and $L(p_0)=0$ (since $\supp(\mu)\subseteq \cZ(p_0)$), which is a contradiction.
This proves Claim 1.\hfill $\blacksquare$\\ 

By Claim 1, we may assume that $\Lambda\cap \partial K=\emptyset$.
Let $S^{\pi}_{K,\odd}$ and $S^{\pi}_{K,\even}$ stand for all $f\in S_K^{\pi}$ of odd and even degree, respectively.
Denoting by $\oplus$ the direct sum of matrices,
we define a linear matrix function 
    $$
    \cL(x,y):=\cL_1(x)\oplus \cL_2(x,y),
    $$
where
    $$
    \cL_1(x):=
    \bigoplus_{f\in S_{K,\odd}^{\pi}}\cH_{f,(\gamma,x)}
    \quad\text{and}\quad
    \cL_2(x,y):=
    \bigoplus_{f\in S^{\pi}_{K,\even}}\cH_{f,(\gamma,x,y)}.
    $$
Let us write
\begin{equation*}
    \cS_{\cL}^{\succ}
    :=\left\{(x,y)\in \RR^2\colon \cL(x,y)\succ 0\right\}
    \quad
    (\text{resp. } 
    \cS_{\cL}^{\succeq}:=\left\{(x,y)\in \RR^2\colon \cL(x,y)\succeq 0\right\}).
\end{equation*}
Let 
   $\proj_x:\RR^2\to\RR$ 
be the projection to the first coordinate, i.e.,        $\proj_x(x,y)=x.$\\

\noindent \textbf{Claim 2:} 
$\proj_x \big(\cS_{\cL}^{\succeq}\big)$ is an interval with a non-empty interior.\\

\noindent\textit{Proof of Claim 2.}
Define 
$\gamma:=(\gamma_0,\gamma_1,\ldots,\gamma_{2k})$,
where $\gamma_i=L(x^i)$ for each $i$.
By the first part of the proof we have that $\cM^{(\fa)}_{L,K,\Lambda}\neq \emptyset$.
Hence, $\gamma$ has an infinite extension 
\begin{equation}
    \label{extension-231023}
        (\gamma,\gamma_{2k+1},\gamma_{2k+2},\ldots)
\end{equation}
generated by moments
of some measure $\mu \in \cM^{(\fa)}_{L,K,\Lambda}$.
We have that
$$\cH_{f,(\gamma,\gamma_{2k+1},\ldots \gamma_{2k+2\ell})}
=\cH_{1,f\cdot (\gamma,\gamma_{2k+1},\ldots \gamma_{2k+2\ell})}\succeq 0$$ 
for each 
$f\in S^{\pi}_K$ and every $\ell\in \NN$.
In particular, $\cL(\gamma_{2k+1},\gamma_{2k+2})\succeq 0$.
Hence, the sets $\cS_{\cL}^{\succeq}$ and $\proj_x \big(\cS_{\cL}^{\succeq}\big)$ 
are non-empty.

Since $\proj_x \big(\cS_{\cL}^{\succeq}\big)$ is a projection of $\cS_{\cL}^{\succeq}$, it is convex and hence an interval.
It remains to prove that $\proj_x \big(\cS_{\cL}^{\succeq}\big)$ is not a singleton. 
Assume on the contrary that 
\begin{equation}
    \label{assumption}
    \proj_x \big(\cS_{\cL}^{\succeq}\big)=\{x_0\}\quad \text{for some }x_0\in \RR. 
\end{equation}
Then $\gamma_{2k+1}$ in \eqref{extension-231023} is uniquely determined and equal to $x_0$.
We separate two cases according to the existence of 
$f\in S^\pi_{K,\odd}$ such that
$
\cH_{f,(\gamma,x_0)}
$ is singular.\\

\noindent\textbf{Case 1:} 
\textit{There exists $f\in S^\pi_{K,\odd}$ such that
$
\cH_{f,(\gamma,x_0)}
=
\cH_{1,f\cdot (\gamma,x_0)}
$
is singular.}

It follows by Theorem \ref{Hamburger} that each $\gamma_{2k+i}$ in \eqref{extension-231023} is uniquely determined
by $\gamma$. 
But then
\begin{equation}
\label{eq:cardinality}    
    \Card(\cM_{L,K,\Lambda})
    =
    \Card(\cM^{(\fa)}_{L,K,\Lambda})=1.
\end{equation}
For $t\in K\setminus \Lambda$ let $\ev_{t}:\RR[x]\to \RR$ be a functional defined on each $x^i$ by $\ev_t(x^i):=t^i$. 
Due to finite dimensionality there exists $\epsilon_{t}>0$ such that
$(L-\epsilon_t \ev_t)|_{T_{S_K}^{(2k)}\setminus\{0\}}>0$.
It follows that $\cM^{(\fa)}_{L-\epsilon_t \ev_t,K,\Lambda}\neq \emptyset$
by the first part of the theorem. Hence, any $t\in K\setminus \Lambda$ is in the support of some measure from $\cM^{(\fa)}_{L,K,\Lambda}$.
Therefore $\Card(\cM_{L,K,\Lambda})=\infty$, which is in contradiction with \eqref{eq:cardinality}. 
So in this case \eqref{assumption} cannot hold.\\

\noindent\textbf{Case 2:} 
\textit{For all $f\in S^\pi_{K,\odd}$ we have that
$
\cH_{f,(\gamma,x_0)}\succ 0.
$}

Let us write
\begin{equation*}
    \cS_{\cL_2}^{\succ}
    :=\left\{(x,y)\in \RR^2\colon \cL_2(x,y)\succ 0\right\}
    \quad
    (\text{resp. } 
    \cS_{\cL_2}^{\succeq}:=\left\{(x,y)\in \RR^2\colon \cL_2(x,y)\succeq 0\right\}).
\end{equation*}
Note that
\begin{align*}
(x_0,y)\in \cS_{\cL_2}^{\succeq}
\quad& \Leftrightarrow \quad 
\cH_{f,(\gamma,x_0,y)}\succeq 0 \quad \text{for all }f\in S_{K,\even}^{\pi},\\
(x_0,y)\in \cS_{\cL_2}^{\succ}
\quad& \Leftrightarrow \quad 
\cH_{f,(\gamma,x_0,y)}\succ 0 \quad \text{for all }f\in S_{K,\even}^{\pi}.
\end{align*}
By the form of $\cH_{f,(\gamma,x_0,y)}$, the solution set of $\cH_{f,(\gamma,x_0,y)}\succ 0$ is an open interval of the form $(a,\infty)$ or $(-\infty,a)$ for some $a\in \RR$ and the solution set of $\cH_{f,(\gamma,x_0,y)}\succeq 0$ is then either $[a,\infty)$ or $(-\infty,a]$. Therefore $\cS_{\cL_2}^{\succeq}:=\{(x_0,y_0)\}$ is either a singleton or
$\proj_y\big(\cS_{\cL_2}^{\succeq}\big):=[y_1,y_2]$ is an interval with a non-empty interior, where 
   $\proj_y:\RR^2\to\RR$ 
is the projection to the second coordinate, i.e., $\proj_y(x,y)=y$. In the first case there exists  $f\in S^\pi_{K,\even}$ such that
$
\cH_{f,(\gamma,x_0,y_0)}
$
is singular and by the same reasoning as in Case 1 above, the equalities \eqref{eq:cardinality} should hold, which leads to a contradiction. In the second case $\proj_y\big(\cS_{\cL_2}^{\succ}\big)=(y_1,y_2)$ and there is $(x_0,y)\in \cS_{\cL_2}^{\succ}$. But then $(x_1,y)\in \cS_{\cL}^{\succ}$ for some $x_1$ close enough to $(x_0,y)$, which is a contradiction with \eqref{assumption}.\\

This proves the Claim 2.
\hfill$\blacksquare$\\

Fix $f\in S^{\pi}_{K,\even}$ and $x_0$ from the interior of
$\proj_x \big(\cS_{\cL}^{\succeq}\big)$.
Let $y_{f,x_0}\in \RR$ be such that $\cH_{f,(\gamma,x_0,y_{f,x_0})}\succeq 0$ and 
$\cH_{f,(\gamma,x_0,y_{f,x_0})}\not\succ 0$. Namely, $y_{f,x_0}$ is uniquely determined by the equality
\begin{equation}
    \label{def-of-gamma-2k+1}
    \widetilde L(fx^{2k+2-\deg f})=
    v^T
    (\cH_{f,\gamma})^{-1}
    v
\end{equation}
as the moment of $x^{2k+2}$,
where 
$$v=\left(L_{f,(\gamma,x_0)}(x^{i})\right)_{i=k+1-\deg f/2}^{2k+1-\deg f}$$
and
$\widetilde L:\RR[x]_{\leq 2k+2}\to \RR$ 
is the extension of $L_{(\gamma,x_0)}:\RR[x]_{\leq 2k+1}\to \RR$.
By Theorem \ref{230422-1853},
the generating polynomial 
$p_{f\cdot (\gamma,x_0,y_{f,x_0})}$
has one of $\lambda\in \Lambda$ as its root only for countably many choices $x_0$.
Thus, as $f$ runs over the set $S^{\pi}_K$ only countably many $x_0$ are such that the generating polynomial $p_{f\cdot (\gamma,x_0,y_{f,x_0})}$
has one of $\lambda\in \Lambda$ as a root. 
So there exists $\widetilde x$ in the interior of $\proj_x \big(\cS_{\cL}^{\succeq}\big)$
such that none of the generating polynomials 
$p_{f\cdot (\gamma,\widetilde x,y_{f,\widetilde x})}$
has some $\lambda\in \Lambda$
as a root. 
Choosing this $\widetilde x$ and the smallest $y_{f,\widetilde x}$ over all $f\in S^\pi_{\even}$ gives $\widetilde \gamma=(\gamma,\widetilde x,y_{f,\widetilde x})$ such that $L_{\widetilde \gamma}$ is $K$--positive. 
Moreover, if there are more choices of $f$, we choose one of the lowest degree.
By \cite[Th\'eor\`eme II, p.\ 129]{Tch57}, we have that $\cM_{L_{\widetilde \gamma},K}\neq \emptyset$.
Since $L_{f,\widetilde \gamma}$ is singular, Theorem \ref{Hamburger} implies that the $K$--rm $\nu$ for $L_{f\cdot \widetilde \gamma}$  
is unique and supported on $\cZ(p_{f\cdot(\gamma,\widetilde x,y_{f,\widetilde x})}).$
Since for every $\mu\in \cM_{L_{\widetilde \gamma},K}$ we have that 
$f\cdot \mu\in \cM_{L_{f\cdot \widetilde \gamma},K}$ (see \eqref{def:new-measure}), it follows that 
$f\cdot \mu=\nu$. Hence, $\supp(\mu)\subseteq \cZ(f)\cup\cZ(p_{f\cdot(\gamma,\widetilde x,y_{f,\widetilde x})})$. 
If $\cZ(f)\not\subseteq \supp(\mu)$, 
then there exists $\widetilde f\in S^{\pi}$ of lower degree than $f$, such that 
$L_{\widetilde f,\widetilde \gamma}$ is also singular (since
$\widetilde f\cdot \mu\in \cM_{L_{\widetilde f\cdot \widetilde\gamma, K}}$). 
But this is a contradiction with the choice of $f$,
whence $\supp(\mu)=\cZ(f)\cup\cZ(p_{f\cdot(\gamma,\widetilde x,y_{f,\widetilde x})})$.
Note that the size of this union is $\deg(f)+\deg p_{f\cdot(\gamma,\widetilde x,y_{f,\widetilde x})}$,
which is at most $2\ell_1+ (\frac{2k-2\ell_1}{2}+1)=k+\ell_1+1$.
This proves \ref{thm:nonsingular-part2} of the moreover part.\\

Let us now prove 
\ref{thm:nonsingular-part1} of the moreover part. Note that 
$\ell_1=0$,
$S^\pi_{K,\even}=\{1\}$ and 
$$S^\pi_{K,\odd}=   
    \left\{
        \begin{array}{rl}   
            \emptyset,& \text{if }K=\RR,\\
            x-a,&   \text{if }K=[a,\infty),\\
            a-x,&   \text{if }K=(-\infty,a].
        \end{array}
    \right.$$
The proof is now verbatim the same to the proof of part \ref{thm:nonsingular-part2}.\\

Next we prove 
\ref{thm:nonsingular-part3} of the moreover part. Assume that $K$ is bounded from above and $b:=\max(K)$.
By the first part of the theorem there exists $\mu\in \cM^{(\fa)}_{L,K,\Lambda}$. Then the support of $\mu$ is contained in $[a,b]$ for some $a\in \Int(K)\setminus \Lambda$ and $a\notin \supp(\mu)$. Let $\widetilde K:=K\cap [a,b]$. Note that $\partial \widetilde K\subseteq \partial K\cup \{a\}$,
$\Card(\partial \widetilde K)\leq 2\ell_1+\ell_2+1$ and $L$ is strictly $\widetilde K$--positive. 
Indeed, since $\supp(\mu) \subseteq \widetilde K$, $L$ is clearly $\widetilde K$--positive.
It remains to show that it is strictly $\widetilde K$--positive. If $L(p)=0$ for some $0\neq p|_{\widetilde K}\geq 0$,
then $p=p_1p_2$ with $\supp(\mu)\subseteq \cZ(p_1)\subset (a,\infty)$ and $\cZ(p_2)\subseteq (-\infty,a]$. From $\cZ(p_1)\subset (a,\infty)$ it follows that
$p_1$ has a constant sign on $(-\infty,a]$ and from $\cZ(p_2)\subseteq (-\infty,a]$ 
it follows that $p_2$ has a constant sign on $(a,\infty)$.
From $p|_{\widetilde K}\geq 0$ and a constant sign of $p_2$ on $\widetilde K$, 
it follows that $p_1$ has constant sign on $\widetilde K$. Multiplying $p_1$ with 
$-1$ if necessary we can assume that $p_1|_{\widetilde K}\geq 0$. Also $L(p_1)=0$,
because $\supp(\mu)\subseteq \cZ(p_1)$. Since $p_1$ does not change sign on both $\widetilde K$ and on $(-\infty,a]$, $a\in \widetilde K$ and $p_1(a)\neq 0$, it follows that $p_1$
does not change the sign on $K$. But then $L(p_1)=0$ for $0\neq p_1|_{K}\geq 0$
and $L$ is not strictly $K$--positive,
which is a contradiction.
By \ref{thm:nonsingular-part2}, part \ref{thm:nonsingular-part3} for $K$ bounded from above follows.
If $K$ is bounded from below, the proof is analogous.\\

Finally we prove 
\ref{thm:nonsingular-part4} of the moreover part.
The proof is analogous to the proof of \ref{thm:nonsingular-part3} above only that $b\neq \max(K)$ but merely $a,b\in \Int(K)\setminus (\Lambda\cup \supp(\mu))$, and hence $\Card(\partial \widetilde K)\leq 2(\ell_1+1)+\ell_2$.
\end{proof}

The following is the solution to the singular case of the $K$--TMP for $L$.

\begin{theorem}[Singular case]
\label{compact-TRMP-extended}
Let $K\subseteq \RR$ be a closed set and $\Lambda\subset \RR$ be a finite or a countable closed set
such that $\Lambda\cap \iso(K)=\emptyset$.
Let $L:\RR[x]_{\leq 2k}\to \RR$ be a linear functional
with
$L|_{T_{S_K}^{(2k)}\setminus\{0\}}\not >0$, where $S_K$ is the natural description of $K$.
The following statements are equivalent:
	\begin{enumerate}
		\item
        \label{rational-solution}
            $\cM_{L,K,\Lambda}\neq \emptyset$.
		\item
          \label{pt2-rational-solution}
            $\cM^{(\fa)}_{L,K,\Lambda}\neq \emptyset$.
        \item 
            \label{pt3-rational-solution}
                The following statements hold:
                \begin{enumerate}
                    \item 
                \label{point-a}
                $L|_{T_{S_K}^{(2k)}\setminus\{0\}}\geq 0$.
                \item
                \label{point-b}
                If:
                \begin{enumerate}
                    \item 
                        $f_0\in S^{\pi}$ is a polynomial of the lowest degree in $S^{\pi}$ such that $L_{f_0}$ is singular,                    
                    \item 
                        $0\neq p_{f_0}$ is a polynomial
                        of the lowest degree such that 
                        $p_{f_0}^2 \in \ker L_{f_0}$,        
                \end{enumerate}
                then 
                \begin{equation}
                    \label{ass:not-poles}
                    \cZ(f_0p_{f_0})\cap \Lambda=\emptyset.
                \end{equation}
                \item 
                \label{point-c}
                If $K$ is unbounded,
                then
                \begin{equation}
                    \label{additional}
                x^dp_{f_0}^2
                \in \ker
                L_{f_0},
                \end{equation}
                where $f_0$, $p_{f_0}$ are
                as in \eqref{point-b} and 
                $d:=2k-\deg (f_0p_{f_0}^2)$.
                \end{enumerate}
	\end{enumerate}

    Moreover, if $\cM^{(\fa)}_{L,K,\Lambda}\neq \emptyset$,
    then the representing measure $\mu$ for $L$ is unique and 
        $$\supp(\mu)=\cZ(f_0)\cup\cZ(p_{f_0}).$$
\end{theorem}

\begin{proof}
The nontrivial part is to prove
$
\eqref{pt3-rational-solution}
\Rightarrow
\eqref{pt2-rational-solution}.$
Let $K_1=K\setminus \Lambda$. Since
$\Lambda\cap \iso(K)=\emptyset$, it follows that 
$\Pos(K)=\Pos(K_1)$. 
Since $L|_{T_{S_K}^{(2k)}}\geq 0$,
it follows by  \cite[Theorem 2.4]{CF08} that $L_1:=L|_{\RR[x]_{\leq 2k-1}}$ has a $K$--rm.
We will prove that $\Card(\cM_{L_1,K})=1$.
Let $\mu\in \cM_{L_1,K}$. 
Since by assumption $0=L_{f_0}(p_{f_0}^2)$, it follows by $K$--positivity of $L$ that     
    $$
        0=L_{f_0}(p_{f_0})=(L_1)_{f_0}(p_{f_0})
        =\int_{K}f_0p_{f_0}d\mu,
    $$
whence $\supp(\mu)\subseteq \cZ(f_0)\cup\cZ(p_{f_0})$. 
If $\supp(\mu)\neq \cZ(f_0)\cup\cZ(p_{f_0})$, then either $\cZ(f_0)\not\subseteq \supp(\mu)$ or  $\cZ(p_{f_0})\not\subseteq \supp(\mu)$. 
In the first case there exists $f_1\in S^{\pi}$ of lower degree than $f_0$, such that
$L_{f_1}$ is also singular, which is a contradiction.
In the second case there exists $p$ of lower degree than $p_{f_0}$, such that
$L_{f_0}(p^2)=0$, which is a contradiction.
Hence, $\supp(\mu)=\cZ(f_0)\cup\cZ(p_{f_0})$ and $\mu\in \cM_{L_1,K}$ is uniquely determined.
By \eqref{ass:not-poles}, it follows that 
$\mu\in \cM^{(\fa)}_{L_1,K, \Lambda}$.
We separate 3 cases according to $K$ and $\deg f_0$.
\begin{itemize}
\item
If $K$ is compact, then, by \cite[Th\'eor\`eme II, p.\ 129]{Tch57},
$L$ has a $K$--rm and hence $\mu$ must also represents $L(x^{2k})$.
\item 
If $K$ is unbounded, then \eqref{additional} ensures $L(x^{2k})$ is also a moment of $\mu$,
whence $\mu\in \cM^{(\fa)}_{L,K,\Lambda}$.
\end{itemize}
This concludes the proof of the theorem.
\end{proof}

\begin{remark}
\label{remark1}
\begin{enumerate}
    \item 
A special case of Theorems \ref{thm:nonsingular} and \ref{compact-TRMP-extended} for compact $K$ with finite $\Lambda$, which comes from $\cL$ (see \eqref{functional-rational}) with only real poles allowed and $k_0=0$, is \cite[Proposition 2]{Cha94}. In states that $\cM_{L,K,\Lambda}\neq \emptyset$ is equivalent to  $L|_{T_{S_{K}}^{(2k)}}\geq 0$.
However, by Example \ref{counterexample} below this equivalence does not hold.
In the proof of \cite[Proposition 2]{Cha94} it is only proved that 
$L$ being $K$--positive implies that $\cM_{L,K}\neq \emptyset$. 
But by Proposition \ref{cL-and-L-measures} above more is needed, namely  
$\cM_{L,K,\Lambda}\neq \emptyset$.
    \item 
    \label{remark-1-pt2}
    If $K$ is unbounded and \eqref{point-a} and \eqref{point-b} of 
    Theorem \ref{compact-TRMP-extended} are satisfied, then
    $L|_{\RR[x]_{\leq 2k-1}}$
    has a $K$--rm $\mu$ vanishing on $\Lambda$,
    while
    $
    L(x^{2k})
    \geq
    \int_K x^{2k}\dd\mu.
    $
    Condition \eqref{point-c}
    characterizes when the equality occurs in this
    inequality.\hfill$\vartriangle$
\end{enumerate}
\end{remark}

The following example demonstrates that $K$--positivity of the functional $L$ is not sufficient for the existence of a rm in the $K$--RTMP.
The \textit{Mathematica} file with numerical computations can be found on the link \url{https://github.com/ZalarA/RTMP_univariate}.

\begin{example}
\label{counterexample}
Let $K=[0,1]$, $\lambda_1=0$, $\lambda_2=1$, 
$\cR^{(4)}=
\left\{
    \frac{f}{x^2(x-1)^2}\colon f\in \RR[x]_{\leq 4}
\right\}$
and $\cL:\cR^{(4)}\to \RR$ a linear functional defined by
\begin{align*}
\cL(1)&=\gamma_{0}^{(0)}:=\frac{1}{48},\quad
\cL\left(\frac{1}{x}\right)=\gamma_{1}^{(1)}:=\frac{1}{24},\quad
\cL\left(\frac{1}{x^2}\right)=\gamma_{2}^{(1)}=\frac{5}{12},\\
\cL\left(\frac{1}{x-1}\right)&=\gamma_{1}^{(2)}:=-\frac{1}{24},\quad
\cL\left(\frac{1}{(x-1)^2}\right)=\gamma_{2}^{(2)}:=\frac{5}{12}.
\end{align*}
The corresponding functional
$
L:\RR[x]_{\leq 4}\to \RR$ is defined by
$$
L(1)=1,\quad
L(x)=\frac{1}{2},\quad
L(x^2)=\frac{5}{12},\quad
L(x^3)=\frac{3}{8},\quad
L(x^4)=\frac{17}{48}.
$$
The localizing Hankel matrices determining whether $L|_{T^{(4)}_{S_{[0,1]}}}\geq 0$ holds
are
\begin{align*}
\cH_{1,\gamma}
&=
\begin{pmatrix}
    1 & \frac{1}{2} & \frac{5}{12}\\[0.3em]
    \frac{1}{2} & \frac{5}{12} & \frac{3}{8}\\[0.3em]
    \frac{5}{12} & \frac{3}{8} & \frac{17}{48}    
\end{pmatrix},
\;
\cH_{x,\gamma}
=
\begin{pmatrix}
    \frac{1}{2} & \frac{5}{12}\\[0.3em] 
    \frac{5}{12} & \frac{3}{8}
\end{pmatrix},
\;
\cH_{1-x,\gamma}
=
\begin{pmatrix}
    \frac{1}{2} & \frac{1}{12}\\[0.3em] 
    \frac{1}{12} & \frac{1}{24}
\end{pmatrix},
\\[0.3em]
\cH_{x(1-x),\gamma}
&=
\begin{pmatrix}
    \frac{1}{12} & \frac{1}{24}\\[0.3em] 
    \frac{1}{24} & \frac{1}{48}
\end{pmatrix}.
\end{align*}
They are all psd with the eigenvalues
$\approx 1.54, 0.22, 0.007$ for $\cH_{1,\gamma}$,
$\approx 0.86, 0.016$ for $\cH_{x,\gamma}$,
$\approx 0.51, 0.027$ for $\cH_{1-x,\gamma}$,
and
$\frac{5}{48}, 0$ for $\cH_{x(1-x),\gamma}$.
Note that $p(x)=x-\frac{1}{2}$ is a column relation of
$\cH_{x(1-x),\gamma}$ and thus the unique measure
for $L$ consists of the atoms $0,1,\frac{1}{2}$
all with densities $\frac{1}{3}$.
Hence, $\cL$ does not have a $K$--rm, even though $L$ is $K$--positive.
\end{example}

\section{Examples}
\label{sec:examples}

 In this section we derive the solution to the strong Hamburger TMP (Corollary \ref{cor:STHMP}) and the TMP on the unit circle (see Theorem \ref{circle-TMP}), and give an example demonstrating the construction of the minimal representing measure as in the proof of Theorem \ref{thm:nonsingular}.\\

A special case of 
Theorem \ref{compact-TRMP-extended}
is the solution to the \textit{strong truncated Hamburger moment problem}.

\begin{corollary}
[{\cite[Theorem 3.1]{Zal22b}}]
\label{cor:STHMP}
Let $\cL$ be a linear functional on
$$
\cR
=\left\{\frac{f}{x^{2k_1}}\colon f\in \RR[x]_{\leq 2k}\right\},
$$
where $k\geq k_1$.
Let
$L:\RR[x]_{\leq 2k}\to \RR$
be a linear functional defined by 
$L(f):=\cL\big(\frac{f}{x^{2k_1}}\big).$
Then $\cL$ has a $\RR$--representing measure 
if and only if the following stament hold:
\begin{enumerate}
    \item $L$ is square--positive.
    \item If $L$ is singular, then:
    \begin{enumerate}
        \item 
        The generating polynomial $p$ of $\cH_{1}$ does not 
        have 0 as its root.
        \item 
        $x^{2k-2\deg p} \cdot p^2\in \ker L$.
    \end{enumerate}
\end{enumerate}
\end{corollary}

\begin{proof}
Take $K=\RR$, $\Lambda=\{0\}$ and use Theorems
\ref{thm:nonsingular} and \ref{compact-TRMP-extended}. Then
$S_\RR=\{1\}$
and
$T_{S_\RR}^{(2k)}=\sum \RR[x]_{\leq k}^2$.
\end{proof}

A special case of 
Theorem \ref{compact-TRMP-extended}
is the solution to the TMP, which can be used to solve the TMP on the unit circle.

\begin{corollary}
\label{circle-rational}
Let $\cL$ be a linear functional on
$
\cR
=\left\{\frac{f}{(x^2+1)^{\ell_1}}\colon f\in \RR[x]_{\leq 2k}\right\}$,
where $k\geq \ell_1$.
Let
$L:\RR[x]_{\leq 2k}\to \RR$ be a linear functional defined by $L(f):=\cL\big(\frac{f}{(x^2+1)^{\ell_1}}\big).$
Then $\cL$ has a $\RR$--representing measure 
if and only if the following statements hold:
\begin{enumerate}
    \item $L$ is square--positive.
    \item
    \label{circle-pt2}
    If $L$ is singular and $p$ 
    is the generating polynomial of $\cH_{1}$,
    then $x^{2k-2\deg p}\cdot p^2\in \ker L$.
\end{enumerate}
\end{corollary}

\begin{proof}
Take $K=\RR$, $\Lambda=\emptyset$ and use Theorems
\ref{thm:nonsingular} and \ref{compact-TRMP-extended}. Then
$S_\RR=\{1\}$
and
$T_{S_\RR}^{(2k)}=\sum \RR[x]_{\leq k}^2$.
\end{proof}

\begin{remark}
\label{circle-remark}
    If $L$ in Corollary \ref{circle-rational} is only square--positive,
    then as in Remark \ref{remark1}.\eqref{remark-1-pt2},
    the restriction $\cL|_{\cR^{(2k-1)}}$,
    where 
    $
    \cR^{(2k-1)}
    =\big\{\frac{f}{(x^2+1)^{\ell_1}}\colon f\in \RR[x]_{\leq 2k-1}\big\}
    $,
    has some $\RR$--rm $\mu$,
    while
    $
    \cL\big(\frac{x^{2k}}{(x^2+1)^{\ell_1}}\big)
    \geq
    \int_\RR \frac{x^{2k}}{(x^2+1)^{\ell_1}}\dd\mu.
    $
    Condition \eqref{circle-pt2} in Corollary \ref{circle-rational}
    characterizes, when the equality occurs in this
    inequality.\hfill$\vartriangle$
\end{remark}

Below we will use Corollary \ref{circle-rational} to derive a solution to the TMP on the unit circle.\\

Let $k\in \NN$ and $\beta\equiv\beta^{ (2k)}=\{\beta_{i,j}\}_{i,j\in \ZZ_+,\; 0\leq i+j\leq 2k}$ be a bivariate sequence of degree $2k$.
The functional $L_{\beta}:\mbb{R}[x,y]_{\leq 2k}\to \RR$ , defined by 
$$
	L_{\beta}(p):=\sum_{\substack{i,j\in \ZZ_+,\\ 0\leq i+j\leq 2k}} a_{i,j}\beta_{i,j},
\quad \text{where}\quad 
p=
	\sum_{\substack{i,j\in \ZZ_+,\\ 0\leq i+j\leq 2k}} a_{i,j}x^iy^j,
$$
is called the \textbf{Riesz functional of $\beta$}. $L_\beta$ is called \textbf{square--positive}
if $L_\beta(p^2)\geq 0$ for every $p\in \RR[x,y]_{\leq k}$.
For $p\in \RR[x,y]$ we denote by 
$\cZ(p)=\{(x,y)\in \RR^2\colon p(x,y)=0\}$
its set of zeros.\\

The solution to the TMP on the unit circle is the following.

\begin{theorem}[{\cite[Theorem 2.1]{CF02}}]
	\label{circle-TMP}
	Let $p(x,y)=x^2+y^2-1$ and 
	$\beta:=\beta^{(2k)}=(\beta_{i,j})_{i,j\in \ZZ_+,i+j\leq 2k}$, where $k\geq 2$. 
	Then the following statements are equivalent:
	\begin{enumerate}
		\item\label{circle-pt1}
		      $\beta$ has a $\mc Z(p)$--representing measure.
		\item\label{circle-pt3}
                $L_\beta$ is square--positive
                %
                and
                the relations $\beta_{2+i,j}+\beta_{i,2+j}=\beta_{i,j}$ hold for every $i,j\in \ZZ_+$ with $i+j\leq 2k-2$.
	\end{enumerate}
\end{theorem}

\begin{proof}
The non-trivial implication is $\eqref{circle-pt3}\Rightarrow
\eqref{circle-pt1}$.
Due to the relations $\beta_{2+i,j}+\beta_{i,2+j}=\beta_{i,j}$ for every $i,j\in \ZZ_+$ with $i+j\leq 2k-2$,
$L_\beta(q)=0$ for every $q\in \RR[x,y]_{\leq 2k}$ 
of the form $q=(x^2+y^2-1)q_1$ with 
$q_1\in \RR[x,y]_{\leq 2k-2}$.
Let
$x(t)=\frac{t^2-1}{t^2+1}$,
$y(t)=\frac{2t}{t^2+1},$
$t\in \RR$,
be a rational parametrization of $\cZ(p)$, 
which is one-to-one and onto $\cZ(p)\setminus \{(1,0)\}.$ We have $\frac{1}{t^2+1}     =\frac{1}{2}(1-x(t))$, $\frac{t}{t^2+1}     =\frac{1}{2}y(t)$ and $\frac{t^2}{t^2+1}
    =\frac{1}{2}(1+x(t))$.
Hence,
\begin{align}
\label{circle-univariate}
\begin{split}
\frac{t^i}{(t^2+1)^{2k}}
    &=
    \left\{
    \begin{array}{rl}
    \frac{1}{2^{2k}}(1-x(t))^{2k-i} (y(t))^i,&
    i=0,\ldots,2k,\\[1em]
    \frac{1}{2^{2k}}(1+x(t))^{i-2k} (y(t))^{4k-i},&
    i=2k+1,\ldots,4k,\\
    \end{array}
    \right.\\[0.5em]
\frac{t^i}{(t^2+1)^{k}}
    &=
    \left\{
    \begin{array}{rl}
    \frac{1}{2^{k}}(1-x(t))^{k-i} (y(t))^i,&
    i=0,\ldots,k,\\[1em]
    \frac{1}{2^{k}}(1+x(t))^{i-k} (y(t))^{2k-i},&
    i=k+1,\ldots,2k.
    \end{array}
    \right.
\end{split}
\end{align}
Let 
$\cR^{(4k)}:=
\left\{
\frac{f(t)}{(t^2+1)^{2k}}
\colon 
f\in \RR[t]_{\leq 4k}
\right\}
$
and
define the functional
$\cL:\cR^{(4k)}\to \RR$
by
\begin{equation}
\label{univariate-moments}
\cL\Big(\frac{t^i}{(t^2+1)^{2k}}\Big)=
\left\{
\begin{array}{rl}
\frac{1}{2^{2k}}
L_{\beta}\big((1-x)^{2k-i}y^i\big),&
i=1,\ldots,2k,\\[1em]
\frac{1}{2^{2k}}L_{\beta}\big((1+x)^{i-2k}y^{4k-i}\big),&
i=2k+1,\ldots,4k.
\end{array}
\right.
\end{equation}
Let 
$L:\RR[t]_{\leq 4t}\to \RR$
be a linear functional defined by 
$L(f):=\cL\big(\frac{f}{(t^2+1)^{2k}}\big)$.
Using correspondences
\eqref{circle-univariate}--\eqref{univariate-moments}
for $g\in\RR[t]_{\leq 2k}$, we have that
$$L(g^2)=
\cL\Big(\frac{g^2}{(t^2+1)^{2k}}\Big)
=
L_{\beta}(
p_1^2)
\geq 0
$$
for some $p_1\in \RR[x,y]_{\leq k}$. 
Hence, $L$ is square--positive.
By Corollary \ref{circle-rational} and Remark \ref{circle-remark}, there exists a $\RR$--rm $\mu$ for
    $\cL|_{\cR^{(4k-1)}}$,
    where 
    $\cR^{(4k-1)}:=\left\{\frac{f(t)}{(t^2+1)^{2k}}\colon
    f\in \RR[t]_{\leq 4k-1}\right\}$,
    while
    $$
    \Delta:=\cL\Big(\frac{t^{4k}}{(t^2+1)^{2k}}\Big)
    -
    \int_K \frac{t^{4k}}{(t^2+1)^{2k}}\dd\mu\geq 0.
    $$
    If $\Delta=0$, then the pushforward measure 
    $\phi_{\#}(\mu)$, where 
    $$\phi:\RR\to \cZ(p)\setminus \{(1,0)\},
    \quad \phi(t)=(x(t),y(t)),
    $$
    is a $\cZ(p)$--rm
    for $L_\beta$.
    Otherwise we add the atom $(1,0)$
    with the density 
    $\Delta$
    to $\phi_{\#}(\mu)$
    and we a get a $\cZ(p)$--rm
    for $L_\beta$.
\end{proof}

\begin{remark}
The proof of Theorem \ref{circle-TMP} is done using the solution to the trigonometric moment problem \cite{CF91}.\hfill$\vartriangle$
\end{remark}

The following example demonstrates the construction of the representing measure for a functional similarly as in the proof of Theorem \ref{thm:nonsingular} (but allowing $\Lambda\cap \partial K\neq \emptyset$).
A minimal $K$--representing measure avoiding real poles does not exist, but allowing one more atom such a measure exists. The \textit{Mathematica} file with numerical computations can be found on the following link \url{https://github.com/ZalarA/RTMP_univariate}.

\begin{example}
\label{counterexample-2}
Let $K=(-\infty,0]\cup [1,2]\cup [3,\infty)$, $\lambda_1=0$, $\lambda_2=1$,
$\lambda_3=2$,
$$\cR^{(6)}=
\left\{
    \frac{f}{x^2(x-1)^2(x-2)^2}\colon f\in \RR[x]_{\leq 6}
\right\}$$
and $\cL:\cR^{(6)}\to \RR$ a linear functional, defined by
\begin{align*}
\cL(1)&=\gamma_{0}^{(0)}:=\frac{1539}{128},\quad
\cL\left(\frac{1}{x}\right)=\gamma_{1}^{(1)}:=\frac{-255}{64},\quad
\cL\left(\frac{1}{x^2}\right)=\gamma_{2}^{(1)}=\frac{235}{32},\\
\cL\left(\frac{1}{x-1}\right)&=\gamma_{1}^{(2)}:=\frac{3}{64},\quad
\cL\left(\frac{1}{(x-1)^2}\right)=\gamma_{2}^{(2)}:=\frac{313}{96},\\
\cL\left(\frac{1}{x-2}\right)&=\gamma_{1}^{(3)}:=\frac{253}{64},\quad
\cL\left(\frac{1}{(x-2)^2}\right)=\gamma_{2}^{(3)}:=\frac{713}{96}.
\end{align*}
The corresponding functional
$
L:\RR[x]_{\leq 6}\to \RR$ is defined by
\begin{align*}
L(1)
&=1,\quad
    L(x)=\frac{13}{12},\quad
    L(x^2)=\frac{23}{8},\quad
    L(x^3)=\frac{307}{48},\quad
    L(x^4)=\frac{555}{32},\\
L(x^5)
&=\frac{9043}{192},\quad
    L(x^6)=\frac{17203}{128}.
\end{align*}
Let $f_1(x)=x(x-1)$, $f_2(x)=(x-2)(x-3)$ and $f_3=f_1f_2$.
The localizing Hankel matrices determining whether $L|_{T^{(6)}_{S_{K}}}\geq 0$ holds are
\begin{align*}
\cH_{1,\gamma}
&=
\begin{pmatrix}
 1 & \frac{13}{12} & \frac{23}{8} & \frac{307}{48} \\[0.3em]
 \frac{13}{12} & \frac{23}{8} & \frac{307}{48} & \frac{555}{32} \\[0.3em]
 \frac{23}{8} & \frac{307}{48} & \frac{555}{32} & \frac{9043}{192} \\[0.3em]
 \frac{307}{48} & \frac{555}{32} & \frac{9043}{192} & \frac{17203}{128} 
\end{pmatrix},
\;
\cH_{f_1,\gamma}
=
\begin{pmatrix}
\frac{43}{24} & \frac{169}{48} & \frac{1051}{96} \\[0.3em]
 \frac{169}{48} & \frac{1051}{96} & \frac{5713}{192} \\[0.3em]
 \frac{1051}{96} & \frac{5713}{192} & \frac{33523}{384} 
\end{pmatrix},
\\[0.3em]
\cH_{f_2,\gamma}
&=
\begin{pmatrix}
    \frac{83}{24} & -\frac{71}{48} & \frac{251}{96} \\[0.3em]
 -\frac{71}{48} & \frac{251}{96} & -\frac{239}{192} \\[0.3em]
 \frac{251}{96} & -\frac{239}{192} & \frac{1139}{384} \\
\end{pmatrix},
\cH_{f_3,\gamma}
=
\begin{pmatrix}
 \frac{131}{32} & -\frac{247}{64} \\[0.3em]
 -\frac{247}{64} & \frac{539}{128} 
\end{pmatrix}.
\end{align*}
They are all pd with the eigenvalues
$\approx 153.6, 1.44, 0.46, 0.12$ for $\cH_{1,\gamma}$,
$\approx 98.9, 0.85, 0.31$ for $\cH_{f_1,\gamma}$,
$\approx 6.74, 1.71. 0.58$ for $\cH_{f_2,\gamma}$,
and
$8.01, 0.29$ for $\cH_{f_3,\gamma}$.
Now we will check that there is no pair $(\gamma_{7},\gamma_{8})\in \RR^2$, 
such that for the extension $\widetilde\gamma=(\gamma,\gamma_{7},\gamma_{8})$ it holds that 
$\cH_{1,\widetilde \gamma}$ is psd and singular, while 
$\cH_{f_1,\widetilde \gamma}$,
$\cH_{f_2,\widetilde \gamma}$,
$\cH_{f_3,\widetilde \gamma}$ 
are psd. Using Schur complements we have:
\begin{align}
\label{lower-bound-gamma-2k+2-v1}
\begin{split}
\cH_{1,\widetilde \gamma}\succeq 0 
&\quad\Leftrightarrow\quad
\gamma_{8}\geq 0 \quad \text{and} \quad \cH_{1,\widetilde \gamma}/\cH_{1,\gamma}\geq 0,\\
\cH_{f_1,\widetilde \gamma}\succeq 0 
&\quad\Leftrightarrow\quad
\gamma_{8}-\gamma_{7}\geq 0 \quad \text{and} \quad  \cH_{f_1,\widetilde \gamma}/\cH_{f_11,\gamma}\geq 0,\\
\cH_{f_2,\widetilde \gamma}\succeq 0 
&\quad\Leftrightarrow\quad
\gamma_{8}-5\gamma_{7}+6\gamma_{6}\geq
0\quad \text{and} \quad 
\cH_{f_2,\widetilde \gamma}/\cH_{f_2,\gamma}\geq 0,\\
\cH_{f_3,\widetilde \gamma}\succeq 0 
&\quad\Leftrightarrow\quad
\gamma_{8}-6\gamma_{7}+11\gamma_{6}-6\gamma_{5}\geq 
0\quad \text{and} \quad 
\cH_{f_3,\widetilde \gamma}/\cH_{f_3,\gamma}\geq 0.
\end{split}
\end{align}
Let $f_0:=1$. We have that
\begin{align}
\label{lower-bound-gamma-2k+2}
\begin{split}
    \cH_{f_0,\widetilde\gamma}/\cH_{f_0,\gamma}
        &=\gamma_8-\frac{220}{591}\gamma_7^2+\frac{15971773 }{56736}\gamma_7-\frac{4733803996639}{87146496},\\[0.2em]
    \cH_{f_1,\widetilde\gamma}/\cH_{f_1,\gamma}
        &=\gamma_8-\frac{11 }{39}\gamma_7^2+\frac{3154553 }{14976}\gamma_7-\frac{6505636110821}{161021952},\\[0.2em]
    \cH_{f_2,\widetilde\gamma}/\cH_{f_2,\gamma}
        &=
\gamma_8-\frac{376 }{369}\gamma_7^2+\frac{13927589 }{17712}\gamma_8-\frac{29105958864401}{190439424},\\[0.2em]
    \cH_{f_3,\widetilde\gamma}/\cH_{f_3,\gamma}
        &=\gamma_8-\frac{131 }{75}\gamma_7^2+\frac{38902817 }{28800}\gamma_7-\frac{11603048263019}{44236800}.
\end{split}
\end{align}
Computation with \textit{Mathematica} shows that there is no choice of $\gamma_{7}\in \RR$ such that for 
$\gamma_{8}=\max(0,\cH_{1,\widetilde \gamma}/\cH_{1,\gamma})$,
we would have $\cH_{f_i,\widetilde \gamma}\geq 0$
$i=1,2,3$. Namely, the conditions in \eqref{lower-bound-gamma-2k+2-v1},\eqref{lower-bound-gamma-2k+2} are of the form 
$\gamma_{8}\geq \max(0,q_i(\gamma_{7}))$,
where $q_i(\gamma_7)$ is a quadratic function in $\gamma_{7}$ corresponding to $f_i$. So the question is whether there exists $\gamma_{7}$ such that $\max(0,q_0(\gamma_{7}))\geq \max(0,q_i(\gamma_{7}))$ for each $i$. Since this is not true, the example shows that there is no minimal representing measure for $L$ that would be supported on $K$ vanishing on poles.

Choosing $\gamma_7=370$ and $\gamma_8=2000$
and repeating the computations above for $\gamma_9$
and $\gamma_{10}$ it turns out that for every 
$\gamma_9\in [-71.50, 845.19]$, the moment matrix $\cH_{f_0}$ corresponding to $f_0$ restricts $\gamma_{10}$ the most from below among all $f_i$, $i=0,1,2,3$. Hence, choosing the smallest possible $\gamma_{10}$ for $\gamma_9$  from this interval gives a representing measure for $\gamma$
supported on the zeroes of the generating polynomial of 
$\cH_{f_0}$.
At most three choices of $\gamma_{9}$ will be such that $\lambda_i$ is one of the zeros of the generating polynomial, so we avoid those.
\end{example}

\end{document}